\documentclass[preprint,12pt]{elsarticle}

\usepackage{amsmath,mathrsfs,ragged2e}
\usepackage{color}
\usepackage{amsfonts}
\usepackage{amscd}
\usepackage{latexsym}
\usepackage{amssymb}  
\usepackage{amsthm}
\usepackage{hyperref}
\allowdisplaybreaks
\usepackage{graphicx, graphics}
\usepackage{subfig,float}
\numberwithin{equation}{section}
\newtheorem{theorem}{Theorem}[section]
\newtheorem{lemma}{Lemma}[section]

\newtheorem{remark}{Remark}[section]

\newcommand{\ut}{u(\cdot,t)}
\newcommand{\vt}{v(\cdot,t)}

\newcommand{\us}{u(\cdot,s)}
\newcommand{\vs}{v(\cdot,s)}

\newcommand{\gu}{{\nabla u}}
\newcommand{\gv}{{\nabla v}}

\newcommand{\lv}{\Delta v}

\newcommand{\mgu}{|\gu|}
\newcommand{\mgv}{|\gv|}

\newcommand{\mlv}{|\Delta v|}

\newcommand{\nut}{\big\|\ut\big\|}
\newcommand{\nvt}{\big\|\vt\big\|}

\newcommand{\nru}{\big\|u\big\|}

\newcommand{\nv}{\big\|v\big\|}

\newcommand{\momega}{|\Omega|}

\newcommand{\ui}{u_0}
\newcommand{\vi}{v_0}

\newcommand{\rsn}{\mathbb{R}^n}
\newcommand{\lis}{{\mathcal{L}^{\infty}}(\Omega)}

\newcommand{\lts}{{\mathcal{L}^{2}}(\Omega)}
\newcommand{\lps}{{\mathcal{L}^{p}}(\Omega)}
\newcommand{\lqs}{{\mathcal{L}^{q}}(\Omega)}

\newcommand{\wsin}{{\mathcal{W}^{1,\infty}}(\Omega)}
\newcommand{\cso}{ {\mathcal{C}^{0}}(\overline{\Omega})}
\newcommand{\csotm}{\mathcal{ C}^{0}\left(\overline{\Omega}\times\left.\left[0,\tmax\right.\right)\right)}
\newcommand{\cts}{ {\mathcal{C}^{2}}(\overline{\Omega})}
\newcommand{\cstotm}{\mathcal{C}^{2,1}\left(\overline{\Omega}\times\left(0,\tmax\right)\right)}

\newcommand{\wsinb}{{\mathcal{W}^{1,\infty}}(\overline{\Omega})}
\newcommand{\lpfpatn}{\mathcal{L}^{\fpatn}}
\newcommand{\lpftpa}{\mathcal{L}^{\frac{2}{p+\alpha}}}

\newcommand{\tmax}{T_{\mathrm{max}}}
\newcommand{\tin}{t_0}

\newcommand{\intts}{\int^t_{s_0}}

\newcommand{\intT}{\int_{s_0}^T}

\newcommand{\ints}{\int_{\Omega}}

\newcommand{\dt}{\frac{\mathrm{d}}{\mathrm{d}t}}

\newcommand{\uop}{(u+1)^p}
\newcommand{\uopo}{(u+1)^{p-1}}
\newcommand{\uopat}{(u+1)^{p+\alpha-2}}
\newcommand{\uopko}{(u+1)^{p+\kappa-1}}
\newcommand{\uopbt}{(u+1)^{p+\beta-2}}
\newcommand{\uopatn}{(u+1)^{p+\alpha+\frac{2}{n}}}
\newcommand{\uofpat}{(u+1)^{\frac{p+\alpha}{2}}}
\newcommand{\fpatn}{\frac{2(p+\alpha+\frac{2}{n})}{p+\alpha}}
\newcommand{\empbm}{e^{-\frac{p+\beta+m}{m}(t-s)}}
\newcommand{\epbm}{e^{\frac{p+\beta+m}{m}t}}

\newcommand{\ust}{u^*}
\newcommand{\vst}{v^*}

\newcommand{\ubt}{\bar{u}}
\newcommand{\vbt}{\bar{v}}

\begin{document}

\begin{frontmatter}
	
\title{Global existence and stability in a class of  chemotaxis systems with lethal interactions, nonlinear diffusion and production}

\author[GS]{Gnanasekaran Shanmugasundaram\corref{cor1}}
\ead{sekaran@nitt.edu}
\affiliation[GS]{organization={Department of Mathematics, National Institute of Technology Tiruchirappalli},
	city={Tiruchirappalli},
	postcode={620015},
	country={India}
}

\author[GS]{Jitraj Saha}
\ead{jitraj@nitt.edu}

\cortext[cor1]{Corresponding author}


\begin{abstract}
	This paper investigates a class of chemotaxis systems describing lethal interactions in a smooth, bounded domain $\Omega \subset \mathbb{R}^n$ subject to homogeneous Neumann boundary conditions.
	 We examine two distinct cases: (i) a fully parabolic system in which both equations are of parabolic type, and (ii) a parabolic–elliptic system consisting of a parabolic first equation coupled with an elliptic second equation. Under appropriate parameter constraints, we establish the existence of a unique globally bounded classical solutions for arbitrary spatial dimensions $n \geq 1$. Additionally, we employ carefully constructed Lyapunov functionals to analyze the long-term behavior of solutions, obtaining rigorous asymptotic stability results.
\end{abstract}


\end{frontmatter}

\section{Introduction and motivation}
\justifying
\quad Chemotaxis refers to the directed movement of cells or organisms in response to chemical gradients. By detecting concentration differences through specialized receptors, cells can navigate toward beneficial compounds or away from harmful ones. This process plays a vital role in essential biological functions, including nutrient foraging and immune defense. Certain microorganisms produce antibiotics as secondary metabolites, which can be toxic to the producers themselves \cite{cundliffe2010}. For example, {\it Escherichia coli (E. coli)} generates {\it Hydrogen peroxide} $(H_2O_2)$ as a metabolic byproduct. During standard growth conditions, redox reactions within \textit{E. coli} cells continuously generate hydrogen peroxide ($H_2O_2$) at a rate of 10–15 µM/s. While scavenging enzymes work to prevent excessive buildup, studies show that $H_2O_2$ levels frequently near toxic levels \cite{ravindra2013, seaver2001}. At high concentrations, $H_2O_2$ becomes lethal to proliferating \textit{E. coli} cells \cite{uhl2015}. A similar dynamic occurs in other antibiotic-producing bacteria. For instance, {\it Bacillus subtilis}, which synthesizes bacilysin, may suffer self-inhibition or death if genetic mutations disrupt its resistance mechanisms. Strategies employed by microorganisms to avoid self-toxicity are detailed in \cite{cundliffe2010}.

In this study, we analyze a chemotaxis model with lethal interactions governed by a system of two partial differential equations. The model describes bacterial movement in response to self-produced hydrogen peroxide ($H_2O_2$), incorporating both random dispersal and repellent-driven migration away from the toxic substance. Furthermore, $H_2O_2$ induces bacterial mortality and is introduced into the system through a predefined source function $f$. Hervas and Negreanu \cite{hervas2025} proposed a lethal-interaction chemotaxis system to model the dual role of $H_2O_2$ in \textit{E. coli} : it acts as both a chemorepellent and a lethal agent, despite being produced by the bacteria. We extend their analysis by introducing nonlinear diffusion and chemical production terms. The resulting system is given by
\begin{align}
	\left\{
	\begin{array}{lll}
		&u_t=d_1 \nabla \cdot(D(u) \nabla u)+\chi\nabla\cdot\left(S(u)\nabla v\right)+r u(1- u^{\kappa -1})-\mu uv, &x\in\Omega,\, t>0,\\
		&\tau v_t=d_2\Delta v+a u^m-b v+f(x,t), &x\in\Omega,\, t>0,\\
		&\frac{\partial u}{\partial\nu}=\frac{\partial v}{\partial\nu}=0, &x\in\partial\Omega, \, t>0,\\
		&u(x,0)=u_0, \quad \tau v(x,0)=\tau v_0, &x\in\Omega,
	\end{array}
	\right.\label{1.1}
\end{align}
in an open, bounded domain $\Omega \subset \mathbb{R}^n$ with smooth boundary $\partial\Omega$, where $\nu$ denotes the outward unit normal vector on $\partial\Omega$ and $\tau\in \{0, 1\}$. The system governs the dynamics of two unknown functions: $u(x,t)$, describing the population density of \textit{E. coli} and $v(x,t)$, representing the concentration of a chemical substance, with an external substance supply represented by $f(x,t)$. The diffusion coefficients $d_1 > 0$ and $d_2 > 0$ characterize the random dispersal of the species and substance, respectively. The chemo-repulsion mechanism is captured through the coefficient $\chi > 0$ in the term $+\chi\nabla\cdot(S(u)\nabla v)$, which models the species' avoidance of chemical gradients. The system includes logistic growth with rate $r > 0$, substance-induced mortality with coefficient $\mu > 0$, while the substance dynamics are governed by production rate $a > 0$ and decay rate $b > 0$. The exponents are assumed to satisfy $\kappa>1$ and $m\geq 1$. We consider non-negative initial data $u_0(x) \geq 0$ and $v_0(x) \geq 0$, with additional regularity assumptions on the initial conditions satisfying
\begin{align}
	\left\{
	\begin{array}{rrll}
		\hspace*{-0.3cm}&\ui \!\!\! &\in\cso,\quad \mbox{with} \quad \ui \geq 0\quad\mbox{in}\,\: \Omega,\\
		\hspace*{-0.3cm}& \tau\vi \!\!\!&\in\wsinb, \quad \mbox{with} \quad  \tau\vi \geq 0\quad\mbox{in}\,\: \Omega.
	\end{array}
	\right.\label{1.2}
\end{align}
Moreover, through out this paper we assume that the positive diffusion functions $D$ and $S$ satisfy the following specific regularity properties 
\begin{align}
D(s)\geq (s+1)^\alpha, \:\: 0\leq S(s)\leq s(s+1)^\beta, \:\: \text{for all}\:\: s\geq 0.\label{1.3}
\end{align}
and
\begin{align}
\quad D, S\in \mathcal{C}^2([0, \infty))\:\: \text{with}\:\: S(0)=0,\label{1.4}
\end{align}
where $\alpha, \beta>0$. Moreover, $f(x,t)$ satisfies
\begin{align}
 0\leq f\leq K, \:\:K>0.\label{1.5}
\end{align}

The mathematical modeling of chemotaxis, the process by which cells direct their movement along chemical gradients, was fundamentally advanced by Keller and Segel in their seminal works \cite{keller1970, keller1971}. Their now-classic model describes this phenomenon through the following system of partial differential equations
\begin{align}
	\left\{
	\begin{aligned}
		u_t =&d_1\Delta u - \chi \nabla\cdot(u\gv),\\ 
		v_t =&d_2\Delta v +\alpha u -\beta v,
	\end{aligned}
	\right. \label{1.6}
\end{align}
where $u(x,t)$ represents the cell population density, $v(x,t)$ denotes the chemoattractant concentration, $d_1, d_2$ are the diffusion coefficients for cells and chemical respectively, $\chi$ quantifies the chemotactic sensitivity, while $\alpha$ and $\beta$ represent the production and degradation rates of the chemoattractant. This minimal yet profound formulation successfully captures the self-organized aggregation patterns observed in microbial systems, particularly in the social amoeba {\it Dictyostelium discoideum}.

The Keller-Segel system and its variants have garnered significant attention due to their wide-ranging applications in biology, medicine, and related disciplines. Over the past decades, extensive theoretical research has been devoted to analyze these models, driven by their relevance to critical biological processes such as bacterial aggregation, immune response dynamics, and cancer metastasis. For a comprehensive overview of the mathematical developments and applications, we refer the articles by Bellomo et al. \cite{bellomo2015}, Horstmann \cite{horstmann2003}, Lankeit and Winkler  \cite{lankeit2020} along with the references therein.

The qualitative behavior of solutions to the Keller-Segel system in a bounded domains exhibits a striking dependence on spatial dimension. As thoroughly documented in Horstmann's review \cite{horstmann2003} regarding system \eqref{1.6} and related models, the solution dynamics vary dramatically with dimension. In one-dimensional settings, Osaki and Yagi \cite{osaki2001} established the global existence of solutions for all initial data. The two-dimensional case presents a more delicate picture: solution behavior is governed by a critical mass threshold, with subcritical initial masses yielding global existence while supercritical masses lead to finite-time blow-up. These foundational blow-up results for the minimal Keller-Segel system \eqref{1.6} emerged through a series of works spanning during 1990s and early 2000s. Subsequent research has developed various model modifications that ensure existence of global classical solution, thereby enabling the study of long-term post-aggregation dynamics.

The classical Keller-Segel model with logistic source of the form,
\begin{align*}
	\left\{
	\begin{aligned}
		u_t= & \Delta u - \chi \nabla\cdot(u\gv) + f(u),\\
		v_t = & \Delta v + u -v,
	\end{aligned}
	\right. 
\end{align*}
has been extensively studied in various contexts. Li and Chen \cite{li2024} established global boundedness and asymptotic stability of classical solutions in whole space ($\rsn, n\geq 2$) for nonnegative initial data, with particular emphasis on long-time behavior. Vance \cite{vance2023} demonstrated, through the test function method, that even sub-logistic source terms suffice to prevent finite-time blow-up in such systems.

The critical role of initial mass conditions was revealed by Liu and Wang \cite{liu2023}, who proved the existence of finite-time blow-up solutions for the Cauchy problem when the initial data satisfies specific mass criteria. Tanaka \cite{tanaka2023} extended these blow-up results to degenerate parabolic-elliptic versions of the system. Complementary to these findings, Jin and Xiang \cite{jin2018} provided a detailed analysis of solution boundedness conditions and investigated the oscillatory behavior patterns in logistic Keller-Segel models.

The Keller-Segel system with external supply of substance,
\begin{align*}
	\left\{
	\begin{aligned}
		u_t = & \Delta u - \nabla\cdot(u\gv), \\
		v_t =& \Delta v +u -v + f(x,t),
	\end{aligned}
	\right. 
\end{align*}
was rigorously analyzed by Black \cite{black2015}. The work establishes three key results: first, the local existence and uniqueness of classical solutions; second, global existence and boundedness of solutions in two spatial dimensions; and third, global existence in higher dimensions under appropriate smallness conditions on both the initial data and the external source term $f(x,t)$.

In \cite{hervas2025}, Herv\'as
and Negreanu analyze a mathematical model
\begin{equation*}
	\left\{
	\begin{aligned}
		u_t =& k\Delta u + \chi \nabla\cdot(u\gv) + ru(1-u)-uv,\\
		v_t =& \Delta v -v +u + f(x,t),
	\end{aligned}
	\right.
\end{equation*}
describing the interaction between a biological species and a chemical substance through motility, negative chemotaxis, and lethality. The authors conducted a rigorous analysis of the system's equilibrium states, with particular focus on characterizing the local stability conditions for spatially homogeneous steady-state solutions.

Liu and Tao \cite{dliu2016} investigated the chemotaxis system
\begin{align*}
	\left\{
	\begin{aligned}
		u_t &= \Delta u - \nabla \cdot (u \nabla v), \\
		v_t &= \Delta v - v + f(u),
	\end{aligned}
	\right.
\end{align*}
where the signal production function \( f \) is nonlinear and satisfies
\( 0 \leq f(s) \leq K s^{\alpha} \) for all \( s > 0 \), with some constants \( K > 0 \) and \( \alpha > 0 \). They proved that if \( 0 < \alpha < \frac{2}{n} \), then, under suitable assumptions on the initial data, the system admits a unique global classical solution which remains uniformly bounded in time for all spatial dimensions \( n \geq 1 \). Subsequently, Winkler~\cite{mwinkler2018} studied the parabolic--elliptic
counterpart
\begin{align*}
	\left\{
	\begin{aligned}
		u_t &= \Delta u - \nabla \cdot (u \nabla v), \\
		0 &= \Delta v - \mu(t) + f(u),
	\end{aligned}
	\right.
\end{align*}
where $
\mu(t) = \frac{1}{|\Omega|} \int_{\Omega} f(u(\cdot,t))$
and \( f \) is a suitably regular function extending the prototype
\( f(u) = u^{\kappa} \) for \( u \geq 0 \) and \( \kappa > 0 \).
It was shown that solutions blow up when \( \kappa > \frac{2}{n} \), whereas for \( \kappa < \frac{2}{n} \) one can always obtain global bounded classical solutions for a wide class of initial data.
Further related results can be found in \cite{sfrassu2021, wwang2019}.

To the best of our knowledge, no existing study has addressed lethal interactions in chemotaxis systems incorporating nonlinear diffusion and nonlinear production mechanisms. The present work is therefore devoted to developing an analytical framework for investigating such systems. Specifically, we aim to identify sufficient conditions that guarantee the global existence of classical solutions to the proposed model within an open, bounded domain endowed with homogeneous Neumann boundary conditions.

In this context, the principal objectives of the study are threefold. First, we establish the existence of globally defined classical solutions to system \eqref{1.1}. Second, we demonstrate that these solutions remain uniformly bounded in $\Omega \times (0, \infty)$ for all spatial dimensions $n \geq 1$. Finally, we analyze the asymptotic behavior of the solutions under appropriate structural assumptions imposed on the system parameters.

The article is structured as follows: Section 2 introduces preliminary lemmas and establishes the local existence of the classical solution. Section 3 is dedicated to proving the global existence and boundedness of the classical solution for system \eqref{1.1}. In Section 4, we provide the asymptotic behaviour of solutions. Finally, the work concludes with a summary.

With this preparation, we state the main theorems as follows.
\begin{theorem}[Global existence of solutions]\label{t.1.1}
Let $\Omega \subset\rsn (n\geq 1)$ be an open, bounded domain with smooth boundary and $\tau\in \{0, 1\}$. Suppose that the constants $d_1, d_2, \chi, r, \mu, a, b$ all are positive, $\kappa>1, m\geq 1$ and the functions $D, S$, $f$ satisfy \eqref{1.3}-\eqref{1.5}. If
\begin{align*}
\beta+m<\alpha+\frac{2}{n},
\end{align*}
then for any nonnegative initial data $(\ui, \vi)$ satisfying \eqref{1.2}, the system \eqref{1.1} admits a unique classical solution $(u, v)$ that is uniformly bounded in the sense that
		\begin{align*}
			\nut_{\lis}+\nvt_{\wsin}\leq C, \quad \text{for all}\:\: t>0,
		\end{align*}
		where the constant $C>0$.
\end{theorem}

\begin{remark}
	While Liu and Tao \cite{dliu2016}, and Winkler \cite{mwinkler2018} focus on pure Keller–Segel type systems with linear diffusion and standard chemotactic flux, Theorem 1.1 applies to nonlinear diffusion–chemotaxis systems with logistic damping. The critical exponent $\frac{2}{n}$ appearing in Liu’s and Winkler’s works is recovered and extended in Theorem 1.1 through the inequality $\beta+m<\alpha+\frac{2}{n}$, which captures the interplay between chemotactic repulsion, nonlinear diffusion and chemical production. The presence of logistic terms in system (1.1) provides an additional stabilizing mechanism, allowing global uniformly bounded solutions under more flexible structural assumptions than those in the classical Keller–Segel framework.
\end{remark}

We study the asymptotic behaviour of the spatially homogeneous steady-states of system \eqref{1.1} when $f$ is a constant supply, we denote it as $f_c$. We now analyze the equilibrium points of the system for $m=1$ and $\kappa\geq 2$. Let $(u, v)$ be a classical solution of \eqref{1.1} satisfying \eqref{1.2} and let $(u_e, v_e)$ denote the equilibrium points of the system \eqref{1.1}, which satisfy the following set of equations
\begin{align*}
	\left\{
	\begin{array}{rrll}
		ru_e(1-u_e^{\kappa-1})-\mu u_ev_e=&0,\\
		a u_e-bv_e+f_c=&0.
	\end{array}
	\right.
\end{align*}
For a constant $f_c$, the system admits two spatially homogeneous steady-states, given by
\begin{align*}
	(u_e, v_e)=\left\{ 
	\begin{array}{rrll}
&(\ust, \vst)=&\left(\frac{br-\mu f_c}{br+a\mu}, \frac{r(a+\mu)}{br+a\mu}\right), \quad &\text{if}\:\: br>\mu f_c,\\
 &(\ubt, \vbt)=&\left(0, \frac{f_c}{b}\right),\quad &\text{if}\:\: br\leq \mu f_c.
	\end{array}
\right.
\end{align*}

\begin{theorem}[Coexistence state]\label{t.2}
	Suppose the assumptions of Theorem \ref{t.1.1} hold. Let the functions $D, S$ satisfy \eqref{1.3}-\eqref{1.4} and $\mu f_c<br$,  $\kappa\geq 2, m=1$. If the parameters satisfy
		\begin{gather*}
		\chi^2<\frac{4d_1d_2\mu}{a\ust}\quad \text{and} \quad 2\beta\leq \alpha
	\end{gather*}
	then the nonnegative classical solution $(u, v, w, z)$ of the system \eqref{1.1} converges exponentially to the coexistence steady-state $(\ust, \vst)$ uniformly in $\Omega$ as $t\to\infty$.
\end{theorem}

\begin{theorem}[\textit{E. coli} vanishing state]\label{t.3}
	Suppose the assumptions of Theorem \ref{t.1.1} hold. Let the functions $D, S$ satisfy \eqref{1.3}-\eqref{1.4} and $\mu f_c\geq br$, $\kappa= 2, m=1$. Then the nonnegative classical solution $(u, v)$ of the system \eqref{1.1} converges exponentially to the \textit{E. coli} vanishing state $(\ubt, \vbt)$ uniformly in $\Omega$ as $t\to\infty$.
\end{theorem}

\section{Preliminaries and local existence}
\quad The aim of this section is to concentrate on the basic estimates and the local existence of the classical solution.
The local well-posedness and extensibility of classical solutions to the chemotaxis system \eqref{1.1} have been rigorously established through the application of an appropriate fixed-point argument combined with standard parabolic regularity theory for quasilinear parabolic systems (see, for instance, \cite{horstmann2005}, \cite{tao2012} and \cite{winkler2010}).
\begin{lemma}[\cite{zheng2015}]\label{l.2.1}
	Let $y$ be a positive absolutely continuous function on $(0, \infty)$  satisfying
	\begin{align*}
		\left\{\hspace*{-0.5cm}
		\begin{array}{rrll}
			&&y'(t)+A y^p\leq B,\\
			&&y(0)=y_0,
		\end{array}
		\right.
	\end{align*}
	with  constants $A>0$, $B\geq 0$ and $p\geq 1$. Then for $t>0$, we have 
	\begin{align*}
		y(t)\leq \max\left\{y_0, \: \left(\frac{B}{A}\right)^\frac{1}{p}\right\}.
	\end{align*}
\end{lemma}

\begin{lemma}[Local Existence]\label{l.2.2}
	Suppose that $\Omega\subset\rsn (n\geq 1)$ is an open, bounded domain with smooth boundary and $\tau \in \{0, 1\}$. Let the initial data $\ui$, $\vi$ satisfy \eqref{1.2}. Moreover,  assume  that functions $D, S$, $f$ satisfy \eqref{1.3}-\eqref{1.5} and that $\kappa>1, m\geq 1$. Then there exists $\tmax\in (0, \infty]$ such that the system \eqref{1.1} admits a unique solution $(u,v)$ satisfying
	\begin{align*}
		u, v \in \csotm\cap \cstotm, \quad \text{for}\:\: \tau=1
	\end{align*}
	or
		\begin{align*}
		u \in \csotm\cap \cstotm \quad \text{and}\quad v\in \mathcal{C}^{2, 0}(\overline{\Omega}), \quad \text{for}\:\: \tau=0.
	\end{align*}
	Moreover, either $\tmax=\infty$ or
	\begin{align}
		\lim_{t\to \tmax}\left(\nut_{\lis}+\nvt_{\wsin}\right)= \infty.\label{l.2.2.1}
	\end{align}
\end{lemma}
\begin{proof}
	The lemma is proved using standard arguments based on parabolic regularity theory. For detailed proofs of analogous results, we refer the reader to Frassu and Viglialoro \cite{sfrassu2021}, Viglialoro and Woolley \cite{gviglialoro2020}, Tao et al. \cite{xtao}, and Gnanasekaran et al. \cite{gnanasekaran2022}. Moreover, the nonnegativity of solutions in $\Omega \times (0, T_{\max})$ follows from the maximum principle in conjunction with the initial condition \eqref{1.2}.
\end{proof}

\begin{lemma}\label{l.2.3}
Let $(u, v)$ be the classical solution of the system \eqref{1.1}, then the solution satisfies
\begin{align}
\ints u\leq M:=\max\left\{\ints u_0,\: \momega\right\}, \quad \forall t\in(0, \tmax).\label{l.2.3.1}
\end{align}
\end{lemma}
\begin{proof}
Integrating the first equation in \eqref{1.1}, we get
\begin{align*}
\dt\ints u \leq r\ints u-r\ints u^\kappa.
\end{align*}
Using reverse H\"older inequality, 
\begin{align*}
\ints u^\kappa\geq \momega^{1-\kappa}\left( \ints u\right)^\kappa.
\end{align*}
Therefore, 
\begin{align*}
\dt \ints u\leq r\ints u-\frac{r}{\momega^{\kappa-1}}\left( \ints u\right)^\kappa.
\end{align*}
Setting $y(t)=\ints u(t)$, then
\begin{align*}
\dt y(t) \leq r y(t)-\frac{r}{\momega^{\kappa-1}} y(t)^\kappa.
\end{align*}
Now, let $z(t)=y(t)^{1-\kappa}$. Therefore, the above inequality takes the form
\begin{align*}
\dt z(t)=(1-\kappa)y^{-\kappa}\dt y(t).
\end{align*}
Finally, we arrive at
\begin{align*}
\dt z(t)+r(\kappa-1)z(t)\leq \frac{r(\kappa-1)}{\momega^{\kappa-1}},
\end{align*}
using Lemma \ref{l.2.1}, we can obtain \eqref{l.2.3.1}.
\end{proof}

\begin{lemma}[Maximal Sobolev regularity \cite{xcao, hieber}]\label{l.2.4}
	Let $\Omega\subset\rsn$ be a bounded domain with smooth boundary,  $q\in(1,\infty)$ and $T>0$. Consider the Neumann problem
	\begin{align*}
		\left\{
		\begin{array}{llll}
			&y_t=\Delta y- y+g,\hspace*{2cm}& x\in \Omega,\: t\in (0, T),\\
			&\frac{\partial y}{\partial\nu}=0,&x\in \partial\Omega,\: t\in (0, T),\\
			&y(x,0)=y_0(x), &x\in\Omega.
		\end{array}
		\right.
	\end{align*}
	If $y_0\in{\mathcal{W}^{2,q}}(\Omega) \:(q>n)$ with $\frac{\partial y_0}{\partial\nu}=0$ on $\partial\Omega$ and any $g\in {\mathcal{L}^{q}}((0,T);{\mathcal{L}^{q}}(\Omega))$, there exists a unique strong solution
	\begin{align*}
		y \in {\mathcal{W}^{1,q}}\left((0,T);{\mathcal{L}^q}(\Omega)\right)\cap{\mathcal{L}^q}\left((0,T); {\mathcal{W}^{2,q}}(\Omega)\right),
	\end{align*}
	satisfying the equation almost everywhere in $\Omega\times (0, T)$. Moreover, there exists $C_q>0$ depending only on $q, n, \Omega$, such that 
	\begin{align*}
		\int_0^T \|y(\cdot,t)\|^q_{\lqs}\mathrm{d}t+&\int_0^T\|y_t(\cdot,t)\|^q_{\lqs}\mathrm{d}t+\int_0^T\|\Delta y(\cdot,t)\|^q_{\lqs}\mathrm{d}t\\
		&\leq C_q\int_0^T\|g(\cdot,t)\|^q_{\lqs}\mathrm{d}t+C_q\|y_0\|^q_{\lqs}+C_q\|\Delta y_0\|^q_{\lqs}.
	\end{align*}
	If $s_0\in(0,T)$ and $y(\cdot, s_0)\in \mathcal{W}^{2,q}(\Omega) \:(q>n)$ with $\frac{\partial y(\cdot, s_0)}{\partial\nu}=0$ on $\partial\Omega$, then
	\begin{align*}
		\intT e^{sq}\|\Delta y(\cdot,t)\|^q_{\lqs}\mathrm{d}t\leq C_q\intT e^{sq}\|g(\cdot,t)\|^q_{\lqs}\mathrm{d}t+C_q \|y(\cdot
		, s_0)\|^q_{\lqs}+C_q\|\Delta y(\cdot, s_0)\|^q_{\lqs}.
	\end{align*}
\end{lemma}

\section{Global existence of solutions}
 We begin by deriving the $\lps$ norm for $u$. 

\subsection{The case $\tau=1$}
For any $s_0 \in (0, \tmax)$ with $s_0 < 1$, Lemma \ref{l.2.2} ensures that $u(\cdot, s_0)$ and $v(\cdot, s_0)$ belong to $\cts$, with $\frac{\partial {v}(\cdot, s_0)}{\partial \nu} = 0$. Let $C > 0$ be a constant such that
\begin{align}
	&\sup\limits_{0\leq s\leq s_0}\big\|\us\big\|_{\lis}\leq C, \quad \sup\limits_{0\leq s\leq s_0}\big\|\vs\big\|_{\lis}\leq C, \quad \big\|\Delta v(\cdot, s_0)\big\|_{\lis}\leq C \label{3.1}
\end{align}
Next, we derive boundedness of $u$ in $t\in (s_0,\tmax)$.

\begin{lemma}\label{l.3.1}
	Suppose that $\Omega \subset \mathbb{R}^n\: (n \geq 1)$, is an open, bounded domain with smooth boundary. Let the functions $D$, $S$, $f$ satisfy \eqref{1.3}--\eqref{1.5} and $\kappa>1, m\geq 1$. If $\beta + m < \alpha + \frac{2}{n}$, then for any $p > 1$, there exists a constant $K_1 > 0$ such that
	\begin{align}
		\|\ut\|_{L^p(\Omega)} \leq K_1, \quad \forall t \in (0, T_{\text{max}}).\label{l.3.1.1}
	\end{align}
\end{lemma}

\begin{proof}
Test the first equation in \eqref{1.1} by $p\uopo$, $p>1$, we get
\begin{align}
\dt \ints (u+1)^p = &-p\ints D(u) \nabla u \cdot\nabla\uopo-p\chi \ints S(u) \nabla v\cdot\nabla \uopo\nonumber\\
&+rp\ints u\uopo-rp\ints u^\kappa \uopo - \mu p\ints uv\uopo\nonumber\\
\leq &-d_1p(p-1)\ints \uopat\,\mgu^2+\chi p(p-1)\ints u\uopbt \gu\cdot \gv\nonumber\\
&+rp\ints (u+1)^p-rp\ints u^\kappa \uopo\nonumber\\
\leq &-\frac{4d_1p(p-1)}{(p+\alpha)^2}\ints \Big|\nabla \uofpat\Big|^2+\chi p(p-1)\ints (u+1)^{p+\beta-1} \gu\cdot \gv\nonumber\\
&+rp\ints (u+1)^p-rp\ints u^\kappa \uopo\nonumber\\
\leq &-\frac{4d_1p(p-1)}{(p+\alpha)^2}\ints \Big|\nabla \uofpat\Big|^2+\frac{\chi p(p-1)}{p+\beta}\ints (u+1)^{p+\beta}\lv\nonumber\\
&+rp\ints (u+1)^p-rp\ints u^\kappa \uopo.\label{l.3.1.2}
\end{align} 
Using Gagliardo-Nirenberg inequality with $C_1>0$ depending only on $p,\alpha,\rho$ and $n$, we get
\begin{align}
\ints \uopatn =&\Big \|\uofpat\Big\|^{\fpatn}_{\lpfpatn}\nonumber\\
\leq& C_1\Big \|\nabla \uofpat\Big\|^{\fpatn \rho}_{\lts} \:\:\Big \| \uofpat\Big\|^{\fpatn (1-\rho)}_{\lpftpa}\nonumber\\
&+C_1 \Big \| \uofpat\Big\|^{\fpatn}_{\lpftpa}\nonumber\\
\leq& C_1\Big\|\nabla \uofpat\Big\|^2_{\lts} \:\:\Big \| \uofpat\Big\|^{\fpatn (1-\rho)}_{\lpftpa}\nonumber\\
&+C_1 \Big \| \uofpat\Big\|^{\fpatn}_{\lpftpa}\nonumber\\
\leq& C_1M^\frac{2}{n}\ints\Big|\nabla \uofpat\Big|^2 +C_1 M^{p+\alpha+\frac{2}{n}},\label{l.3.1.3}
\end{align}
where $\rho=\frac{(p+\alpha)(1-\frac{1}{p+\alpha+\frac{2}{n}})}{p+\alpha+\frac{2}{n}-1}\in (0, 1)$.
Rewrite the above inequality \eqref{l.3.1.3} as follows
\begin{align}
- \ints\Big|\nabla \uofpat\Big|^2 \leq- \frac{1}{C_1M^\frac{2}{n}}\ints \uopatn+M^{p+\alpha}.\label{l.3.1.4}
\end{align}
Using Young's inequality with $C_2>0$ depending only on $\chi, \beta, m, p$ for the second term in \eqref{l.3.1.2}, we get
\begin{align}
\frac{\chi p(p-1)}{p+\beta}\ints (u+1)^{p+\beta}\lv\leq \ints (u+1)^{p+\beta+m}+C_2\ints \mlv^{\frac{p+\beta+m}{m}}.\label{l.3.1.5}
\end{align}
Consider the last term in \eqref{l.3.1.2}
\begin{align}
-rp\ints u^\kappa \uopo \leq& -rp\ints \uopo\left(\frac{1}{2^{\kappa-1}}(u+1)^\kappa-1\right)\nonumber\\
\leq& -\frac{rp}{2^{\kappa-1}}\ints \uopko+rp\ints \uopo.\label{l.3.1.6}
\end{align}
Using Young's inequality to the second term in \eqref{l.3.1.6} with $C_1',C_2'>0$ depending only on $r,\kappa$ and $p$, we get
\begin{align*}
rp\ints \uopo\leq \frac{rp}{2^\kappa}\ints \uopko+C_1'
\end{align*}
and for third term in \eqref{l.3.1.2}
\begin{align}
rp\ints (u+1)^p\leq \frac{rp}{2^\kappa}\ints \uopko+C_2'.\label{l.3.1.7}
\end{align}
Substituting \eqref{l.3.1.4}-\eqref{l.3.1.7} in to \eqref{l.3.1.2}, we get
\begin{align*}
\dt \ints (u+1)^p\leq &-\frac{4d_1p(p-1)}{C_1M^\frac{2}{n}(p+\alpha)^2} \ints \uopatn+\ints (u+1)^{p+\beta+m}+C_2\ints \mlv^{\frac{p+\beta+m}{m}}\\
&-\frac{rp}{2^{\kappa-1}}\ints \uopko+\frac{2rp}{2^\kappa}\ints \uopko+C_3,
\end{align*}
where $C_3:=C_1'+C_2'+\frac{4d_1p(p-1)M^{p+\alpha}}{(p+\alpha)^2}$.
Simplifying, we get
\begin{align}
	\dt \ints (u+1)^p\leq &-\frac{4d_1p(p-1)}{C_1M^\frac{2}{n}(p+\alpha)^2} \ints \uopatn+\ints (u+1)^{p+\beta+m}+C_2\ints \mlv^{\frac{p+\beta+m}{m}}\nonumber\\
	&+C_3.\label{l.3.1.8}
\end{align}
Adding $\frac{p+\beta+m}{m}\ints \uop$ and then multiplying $e^{\frac{p+\beta+m}{m}t}$ on both sides of \eqref{l.3.1.8}, we obtain
\begin{align*}
	\dt\left(\epbm \ints (u+1)^p\right)\leq &-\frac{4d_1p(p-1)}{C_1M^\frac{2}{n}(p+\alpha)^2} \epbm \ints \uopatn\\
	&+\epbm\ints (u+1)^{p+\beta+m}+C_2\epbm\ints \mlv^{\frac{p+\beta+m}{m}}\\
	&+\frac{p+\beta+m}{m}\epbm\ints \uop+C_3\epbm.
\end{align*}
Let $s_0 \in (0, \tmax)$ with $s_0 < 1$, integrating the above inequality over $(s_0, t)$, we arrive at
\begin{align}
\ints \uop\leq &-\frac{4d_1p(p-1)}{C_1M^\frac{2}{n}(p+\alpha)^2} \intts\ints\empbm \uopatn\nonumber\\
&+\intts\ints\empbm (u+1)^{p+\beta+m}+C_2\intts\ints\empbm \mlv^{\frac{p+\beta+m}{m}}\nonumber\\
&+\frac{p+\beta+m}{m}\intts\ints\empbm \uop+C_3\intts\empbm \nonumber\\
&+\ints(u(\cdot, s_0)+1)^p\nonumber\\
\leq &-\frac{4d_1p(p-1)}{C_1M^\frac{2}{n}(p+\alpha)^2} \intts\ints\empbm \uopatn\nonumber\\
&+\intts\ints\empbm\ints (u+1)^{p+\beta+m}+C_2\intts\ints\empbm \mlv^{\frac{p+\beta+m}{m}}\nonumber\\
&+\frac{p+\beta+m}{m}\intts\ints\empbm \uop+C_4,\label{l.3.1.9}
\end{align}
where $C_4=\frac{C_3 m}{p+\beta+m}+\ints(u(\cdot, s_0)+1)^p$. Adopting the maximal Sobolev regularity, we have
\begin{align}
C_2\intts\ints\empbm \mlv^{\frac{p+\beta+m}{m}}\leq& C_3'\intts\ints\empbm\Big(a u^m+f(x,s)\Big)^\frac{p+\beta+m}{m}\nonumber\\
&+C_3' \Big\|v(\cdot, s_0)\Big\|^\frac{p+\beta+m}{m}_{\mathcal{W}^{2, \frac{p+\beta+m}{m}}}\nonumber\\
\leq& C_3'\left(2^{\frac{p+\beta+m}{m}-1}\right)\intts\ints\empbm a^\frac{p+\beta+m}{m} u^{p+\beta+m}\nonumber\\
&+C_3'\left(2^{\frac{p+\beta+m}{m}-1}\right)\intts\ints\empbm f(x,t)^\frac{p+\beta+m}{m}\nonumber\\
&+C_3' \Big\|v(\cdot, s_0)\Big\|^\frac{p+\beta+m}{m}_{\mathcal{W}^{2, \frac{p+\beta+m}{m}}}\nonumber\\
\leq& C_5\intts\ints\empbm u^{p+\beta+m}+C_4',\label{l.3.1.10}
\end{align}
where $C_5=C_3'\left(2^{\frac{p+\beta+m}{m}-1}\right)a^\frac{p+\beta+m}{m}$ and\\ $C_4'=C_3'\left(2^{\frac{p+\beta+m}{m}-1}\right)K^\frac{p+\beta+m}{m}$$\frac{m}{p+\beta+m}\momega+C_3' \Big\|v(\cdot, s_0)\Big\|^\frac{p+\beta+m}{m}_{\mathcal{W}^{2, \frac{p+\beta+m}{m}}}$. Substitute \eqref{l.3.1.10} in to \eqref{l.3.1.9}, we get
\begin{align}
	\ints \uop\leq &-\frac{4d_1p(p-1)}{C_1M^\frac{2}{n}(p+\alpha)^2} \intts\ints\empbm \uopatn\nonumber\\
	&+(1+C_5) \intts\ints\empbm u^{p+\beta+m}\nonumber\\
	&+\frac{p+\beta+m}{m}\intts\ints\empbm \uop+C_{6},\label{l.3.1.11}
\end{align}
where $C_{6}=C_4+C_4'$. Using Young's inequality with $\beta + m < \alpha + \frac{2}{n}$, we obtain
\begin{align*}
(1+C_5)\ints u^{p+\beta+m}\leq \frac{2d_1p(p-1)}{C_1M^\frac{2}{n}(p+\alpha)^2}\ints \uopatn+C_5'
\end{align*}
and
\begin{align*}
\frac{p+\beta+m}{m}\ints \uop\leq \frac{2d_1p(p-1)}{C_1M^\frac{2}{n}(p+\alpha)^2}\ints \uopatn+C_6'.
\end{align*}
Substitute the above two estimates into \eqref{l.3.1.11}, finally  we arrive at
\begin{align*}
	\ints \uop\leq &C_{7}, \qquad \forall t\in(s_0, \tmax),
\end{align*}
where $C_{7}>0$. Considering \eqref{3.1}, the proof is complete.
\end{proof}

\subsection{The case $\tau=0$}

\begin{lemma}\label{l.3.2}
	Let $\Omega \subset \mathbb{R}^n$ ($n \geq 1$) be an open, bounded domain with smooth boundary. Assume that $D$, $S$, $f$ satisfy \eqref{1.3}--\eqref{1.5} and $\kappa > 1, m\geq 1$. If $\beta + m < \alpha + \frac{2}{n}$, then for any $p > 1$, there exists a constant $K_2 > 0$ such that
	\begin{align}
		\|\ut\|_{L^p(\Omega)} \leq K_2, \quad \text{for all } t \in (0, T_{\max}). \label{l.3.2.1}
	\end{align}
\end{lemma}

\begin{proof}
From \eqref{l.3.1.2} and then use the second equation in \eqref{1.1} with $\tau=0$, we have
\begin{align}
	\dt \ints (u+1)^p +\ints (u+1)^p\leq &-\frac{4d_1p(p-1)}{(p+\alpha)^2}\ints \Big|\nabla \uofpat\Big|^2\nonumber\\
	&-\frac{\chi p(p-1)}{p+\beta}\ints (u+1)^{p+\beta}\lv+(1+rp)\ints (u+1)^p\nonumber\\
	&-rp\ints u^\kappa \uopo.\nonumber\\
	\leq & -\frac{4d_1p(p-1)}{(p+\alpha)^2}\ints \Big|\nabla \uofpat\Big|^2\nonumber\\
	&+\frac{\chi ap(p-1)}{d_2(p+\beta)}\ints (u+1)^{p+\beta+m}\nonumber\\
	&+\frac{\chi ap(p-1)K}{d_2(p+\beta)}\ints (u+1)^{p+\beta}+(1+rp)\ints (u+1)^p\nonumber\\
	&-rp\ints u^\kappa \uopo.\label{l.3.2.2}
\end{align} 
Recall \eqref{l.3.1.4}, \eqref{l.3.1.6} and \eqref{l.3.1.7}, we obtain
\begin{align}
	\dt \ints (u+1)^p +\ints (u+1)^p\leq & -\frac{4d_1p(p-1)}{C_1M^\frac{2}{n}(p+\alpha)^2} \ints \uopatn\nonumber\\
	&+\frac{\chi ap(p-1)}{d_2(p+\beta)}\ints (u+1)^{p+\beta+m}\nonumber\\
	&+\frac{\chi ap(p-1)K}{d_2(p+\beta)}\ints (u+1)^{p+\beta}+C_2.\label{l.3.2.3}
\end{align} 
where $C_2>0$ depending only on $M, \alpha, r, \kappa$ and $p$. Now using Young's inequality with $\beta + m < \alpha + \frac{2}{n}$, we get
\begin{align}
\frac{\chi ap(p-1)}{d_2(p+\beta)}\ints (u+1)^{p+\beta+m}\leq \frac{2d_1p(p-1)}{C_1M^\frac{2}{n}(p+\alpha)^2}\ints \uopatn+C_3. \label{l.3.2.4}
\end{align}
and again Young's inequality with $\beta < \alpha + \frac{2}{n}$, gives
\begin{align}
\frac{\chi ap(p-1)K}{d_2(p+\beta)}\ints (u+1)^{p+\beta}\leq \frac{2d_1p(p-1)}{C_1M^\frac{2}{n}(p+\alpha)^2}\ints \uopatn+C_4, \label{l.3.2.5}
\end{align}
where $C_3, C_4>0$ depending only on $d_2, \chi, a, \alpha, \beta, m, \kappa, p$ and $K$. Substituting \eqref{l.3.2.4} and \eqref{l.3.2.5} into \eqref{l.3.2.3}, we arrive at
\begin{align*}
	\dt \ints (u+1)^p +\ints (u+1)^p\leq &C_5,
\end{align*} 
where $C_5>0$. Applying Lemma \ref{l.2.1}, the proof is complete.
\end{proof}

{\bf Proof of Theorem \ref{t.1.1}.} 
For $\tau=1$,  the result follows from the standard parabolic regularity (Ladyzhenskaya et al. \cite{ladyzen}, Amann \cite{amann1995}) applied to the second equation in \eqref{1.1},
\begin{align*}
	v_t-d_2\Delta v+bv=a u^m+f(x,t)
\end{align*}
which ensures the boundedness and regularity of $v$, provided that $u$ satisfies the $\lps$-bound obtained in Lemma \ref{l.3.1}. Hence, there exists a constant $C_1>0$ such that
\begin{align}
	\nvt_{\mathcal{W}^{1,\infty}(\Omega)}\leq C_1, \qquad \forall t\in(0, \tmax).\label{t.1.1.1}
\end{align}
For the case $\tau=0$, \ref{t.1.1.1} comes from the standard elliptic regularity theory. Finally, one can employ the well-known Moser-Alikakos iteration technique (\cite{tao2012} Lemma A.1) with Lemma \ref{l.3.1} and \ref{l.3.2} to prove that there exists $C_2>0$ fulflling
\begin{align*}
	\nut_{\lis}\leq C_2, \qquad \forall t\in(0, \tmax).
\end{align*}
This completes the proof.

\section{Global asymptotic stability}
\quad This section examines the long-term dynamics of solutions to system \eqref{1.1} through Lyapunov functional techniques. Our asymptotic analysis extends the methodology developed in \cite{xbai}. In this section, we assume $f_c$ is a constant supply and $m=1$.
\begin{lemma}\label{l.4.1}
	Let $(u, v)$ be the nonnegative classical solution of the system \eqref{1.1} and suppose that the assumptions of Theorem \ref{t.1.1} hold true. Then there exists $\theta\in(0, 1)$ and $C_1>0$ such that
	\begin{align*}
		\big\|u\big\|_{\mathcal{C}^{2+\theta, 1+ \frac{\theta}{2}}(\overline{\Omega}\times[t, t+1])}+\big\|v\big\|_{\mathcal{C}^{2+\theta, 1+ \frac{\theta}{2}}(\overline{\Omega}\times[t, t+1])}\leq C_1
	\end{align*}
	for all $t\geq 1$.
\end{lemma}
\begin{proof}
	The proof is based on the standard parabolic regularity theory in \cite{ladyzen} and Theorem \ref{t.1.1}. For further details, see \cite{tli, mmporzio}
\end{proof}

\quad Lemma \ref{l.4.1} shows that the solutions $u$ and $v$ are continuous and bounded in $\overline{\Omega} \times [t, t+1]$, with existing first derivatives $\nabla u$ and $\nabla v$ that are H\"older continuous. Furthermore, the uniform boundedness of the H\"older norms implies that the gradients are uniformly bounded in $L^\infty(\Omega)$. Consequently, we deduce that
\begin{align}
	\nru_{\wsin}, \nv_{\wsin} \leq C_2,\label{l.4.1.1}
\end{align}
where $C_2$ is a constant independent of $t$.

\begin{lemma}[\cite{mhirata}]\label{l.4.2}
	Suppose that $\mathcal{F}:(1, \infty)$ is a uniformly continuous nonnegative function such that 
	\begin{align*}
		\int_{1}^\infty \mathcal{F}(t)\: \mathrm{d}t<\infty.
	\end{align*} 
	Then, $\mathcal{F}(t)\to 0$ as $t\to\infty$.
\end{lemma}

First we start with the coexistent state of the species.

\subsection{Coexistence state}
Here we assume that $\mu f_c<br$ hold.
Let $(u, v)$ be the classical solution of \eqref{1.1} satisfying \eqref{1.2} and $(\ust, \vst)$ be the coexistence steady-state of the system \eqref{1.1}.

\begin{lemma}\label{l.4.3}
	Suppose the assumptions of Theorem 1.1 hold with $\tau \in \{0, 1\}$. Let the functions $D$, $S$ satisfy \eqref{1.3}--\eqref{1.4} and $\kappa\geq 2$. If
	\begin{gather*}
		\chi^2 < \frac{4d_1d_2\mu}{a u^*} \quad \text{and} \quad 2\beta \leq \alpha,
	\end{gather*}
	then the solution has the following asymptotic behavior
	\begin{align}
		\big\|\ut - u^*\big\|_{L^\infty(\Omega)} + \big\|\vt - v^*\big\|_{L^\infty(\Omega)} \to 0  \label{l.4.3.1}
	\end{align}
	as $t \to \infty$.
\end{lemma}

\begin{proof}
To analyze the coexistence state, we now introduce the following energy functional
\begin{align*}
	\mathcal{E}_1(t):=a\ints\left(u-\ust-\ust \ln\frac{u}{\ust}\right)
	+\tau\frac{\mu}{2}\ints\left(v-\vst\right)^2	=\mathcal{I}_{11}(t)+\mathcal{I}_{12}(t), \quad t>0.
\end{align*}
It is straightforward to verify that $\mathcal{E}_1(t)\geq 0$. Now compute
\begin{align}
	\dt \mathcal{I}_{11}
	\leq & -d_1\ust\ints D(u) \frac{\mgu^2}{u^2} - \chi\ust \ints S(u) \frac{\gu\cdot\gv}{u^2}\nonumber\\
	& +\ints (u-\ust)(r-ru^{\kappa-1}-\mu v-(r-ru^{*\kappa-1}-\mu \vst)) \nonumber\\
	= & -d_1\ust\ints D(u) \frac{\mgu^2}{u^2}-2\ust\ints \frac{\sqrt{d_1D(u)}\:\: \gu}{u}\:\:\frac{\chi S(u)\gv}{2u\sqrt{d_1D(u)}} -\frac{\chi^2\ust}{4d_1}\ints\frac{S^2(u)\mgv^2}{D(u)u^2}\nonumber\\
	&+\frac{\chi^2\ust}{4d_1}\ints\frac{S^2(u)\mgv^2}{D(u)u^2} +\ints (u-\ust)(-ru^{\kappa-1}-\mu v+ru^{*\kappa-1}+\mu \vst)\nonumber\\
	= & -\ust\ints\left(\frac{\sqrt{d_1D(u)}\:\: \gu}{u}+\frac{\chi S(u)\gv}{2u\sqrt{d_1D(u)}} \right)^2+\frac{\chi^2\ust}{4d_1}\ints\frac{S^2(u)\mgv^2}{D(u) u^2}\nonumber\\
	&-r\ints (u-\ust)(u^{\kappa-1}-u^{*\kappa-1})-\mu\ints (u-\ust)(v-\vst).\label{l.4.3.2}
\end{align}
Consider the second term in the above estimate with use of \eqref{1.3}, we arrive at
\begin{align*}
\frac{\chi^2\ust}{4d_1}\ints\frac{S^2(u)\mgv^2}{D(u) u^2}\leq \frac{\chi^2\ust}{4d_1}\ints\frac{u^2(u+1)^{2\beta}\mgv^2}{(u+1)^{\alpha} \:u^2}\leq \frac{\chi^2\ust}{4d_1}\ints (u+1)^{2\beta-\alpha}\mgv^2.
\end{align*}
Since $2\beta\leq \alpha$, we have
\begin{align}
\frac{\chi^2\ust}{4d_1}\ints\frac{S^2(u)\mgv^2}{D(u) u^2}\leq& \frac{\chi^2\ust}{4d_1}\ints \mgv^2.\label{l.4.3.3}
\end{align}
Substitute \eqref{l.4.3.3} into \eqref{l.4.3.2} and use $(u-\ust)(u^{\kappa-1}-u^{*\kappa-1})\geq (u-\ust)^2$ due to $\kappa\geq 2$, we get
\begin{align}
	\dt \mathcal{I}_{11}	\leq & \frac{\chi^2\ust}{4d_1}\ints \mgv^2-r\ints (u-\ust)^2-\mu\ints (u-\ust)(v-\vst).\label{l.4.3.4}
\end{align}
Similarly, for $\tau=1$
\begin{align}
\dt \mathcal{I}_{12}
=&\ints (v-\vst)(d_2 \lv+a(u-\ust)-b(v-\vst))\nonumber\\
=&-d_2\ints \mgv^2+a\ints(u-\ust)(v-\vst)-b\ints(v-\vst)^2.\label{l.4.3.5}
\end{align}
Adding \eqref{l.4.3.4} and \eqref{l.4.3.5}, we get
\begin{align*}
	\dt \mathcal{E}_1(t)	\leq & \frac{a\chi^2\ust}{4d_1}\ints \mgv^2-ar\ints (u-\ust)^2-a\mu\ints (u-\ust)(v-\vst)-d_2\mu\ints \mgv^2\\
	&+a\mu\ints(u-\ust)(v-\vst)-b\mu\ints(v-\vst)^2\\
	\leq& -\left(d_2\mu-\frac{a\chi^2\ust}{4d_1}\right)\ints \mgv^2-ar\ints (u-\ust)^2-b\mu\ints(v-\vst)^2.
\end{align*}
Simplifying, 
\begin{align}
	\dt \mathcal{E}_1(t)	\leq& -ar\ints (u-\ust)^2-b\mu\ints(v-\vst)^2\leq -\zeta_1 f_1(t),\label{l.4.3.6}
\end{align}
where $\zeta_1>0$ and
\begin{align*}
	f_1(t)=\ints (u-\ust)^2+\ints(v-\vst)^2.
\end{align*}
Upon integrating with respect to $t$, we obtain
\begin{align*}
	\int_1^\infty f_1(t)\:\mbox{d}t \leq \frac{1}{\zeta_1}\Big(\mathcal{E}_1(1)-\mathcal{E}_1(t)\Big)<\infty.
\end{align*}
Since $f_1(t)$ is uniformly continuous in $(1, \infty)$, we use Lemma \ref{l.4.2}, which gives
\begin{align*}
	\ints(\ut-\ust)^2+\ints(\vt-\vst)^2\to 0
\end{align*}
as $t\to\infty$. Applying the Gagliardo-Nirenberg inequality, we have
\begin{align}
	\Big\|\ut-\ust\Big\|_{\lis}\leq C_1\Big\|\ut-\ust\Big\|^{\frac{n}{n+2}}_{{ \mathcal{W}^{1,\infty}}(\Omega)}\:\:\Big\|\ut-\ust\Big\|^{\frac{2}{n+2}}_{\lts}, \quad t>0.\label{l.4.3.7}
\end{align}
From \eqref{l.4.1.1}, we can deduce that $\ut$ converges to $\ust$ in $\lis$ when $t\to\infty$. By using a similar argument, we can derive \eqref{l.4.3.1}. For the case $\tau=0$, we have
\begin{align}
0=&-d_2\mu\ints \mgv^2+a\mu\ints (u-\ust)(v-\vst)-b\mu\ints (v-\vst)^2.\label{l.4.3.8}
\end{align}
Adding \eqref{l.4.3.4} and \eqref{l.4.3.8}, we get same as \eqref{l.4.3.6}. By a similar arguments, we complete the proof.
\end{proof}

\subsection{\textit{E. coli} vanishing state}
Here we assume that $\mu f_c\geq br$ hold.
Let $(u, v)$ be the classical solution of \eqref{1.1} satisfying \eqref{1.2} and $(\ubt, \vbt)$ be the \textit{E. coli} vanishing state of the system \eqref{1.1}.

\begin{lemma}\label{l.4.4}
	Assume the hypotheses of Theorem 1.1 hold with $\tau \in \{0,1\}$. Let the functions $D$, $S$ satisfy \eqref{1.3}--\eqref{1.4} and $\kappa=2$. Then the solution exhibits the following asymptotic behavior
	\begin{align}
\big\|\ut\big\|_{\lis}+\big\|\vt-\vbt\big\|_{\lis}\to 0\label{l.4.4.1}
	\end{align}
	as $t \to \infty$.
\end{lemma}

\begin{proof}
To analyze the \textit{E. coli} vanishing state, we now introduce the following energy functional
\begin{align*}
	\mathcal{E}_2(t):=&a\ints u
	+\tau\frac{\mu}{2}\ints\left(v-\vbt\right)^2=\mathcal{I}_{21}(t)+\mathcal{I}_{22}(t), \quad t>0.
\end{align*}
It is easy to see that $\mathcal{E}_2(t)\geq 0$. Compute
\begin{align}
\dt \mathcal{I}_{21}(t)=&\ints u(r-ru-\mu v)+\mu\ints u\vbt-\mu \ints u\vbt\nonumber\\
=&-\mu\vbt\ints u+r\ints u-r\ints u^2-\mu \ints u(v-\vbt)\nonumber\\
=&-(\mu\vbt-r)\ints u-r\ints u^2-\mu \ints u(v-\vbt)\label{l.4.4.2}
\end{align}
and for $\tau=1$,
\begin{align}
	\dt \mathcal{I}_{22}=&-d_2\ints \mgv^2+a\ints u(v-\vbt)-b\ints(v-\vbt)^2.\label{l.4.4.3}
\end{align}
Adding \eqref{l.4.4.2} and \eqref{l.4.4.3}, we get
\begin{align*}
	\dt \mathcal{E}_{2}(t)=&-a(\mu\vbt-r)\ints u-ar\ints u^2-a\mu \ints u(v-\vbt)-d_2\mu\ints \mgv^2+a\mu\ints u(v-\vbt)\nonumber\\
	&-b\mu\ints(v-\vbt)^2.
\end{align*}
Therefore,
\begin{align}
	\dt \mathcal{E}_{2}(t)\leq &-ar\ints u^2-b\mu\ints(v-\vbt)^2-a(\mu\vbt-r)\ints u\nonumber\\
	\leq& -\zeta_2 f_2(t)-a(\mu\vbt-r)\ints u,\label{l.4.4.4}
\end{align}
where $\zeta_2>0$ and
\begin{align*}
	f_2(t)=\ints u^2+\ints(v-\vbt)^2.
\end{align*}
By a similar arguments as in Lemma \ref{l.4.3}, we derive \eqref{l.4.4.1}. For the case $\tau=0$, we have
\begin{align*}
	0=&-d_2\mu\ints \mgv^2+a\mu\ints u(v-\vbt)-b\mu\ints (v-\vbt)^2.
\end{align*}
By the similar arguments, we complete the proof.
\end{proof}

The following proofs of the main theorems are inspired by \cite{xbai}.

\subsection{Proof of the main theorems}

{\bf Proof of Theorem \ref{t.2}.}
Let $\mathcal{H}(u)=u-\ust \ln  u$, for $u>0$. By applying L'Hôpital's rule, we obtain
\begin{align*}
	\lim\limits_{u-\ust}\frac{\mathcal{H}(u)-\mathcal{H}(\ust)}{(u-\ust)^2}=\frac{1-\frac{\ust}{u}}{2(u-\ust)}=\frac{1}{2\ust}.
\end{align*}
We can therefore select $t_0 > 0$ and using the Taylor's expansion, such that
\begin{align*}
	\ints\left(u-\ust-\ust \ln\frac{u}{\ust}\right)=\ints(\mathcal{H}(u)-\mathcal{H}(\ust))
	\leq\frac{1}{2\ust}\ints(u-\ust)^2, \quad \forall\, t>t_0.
\end{align*}
From the above estimate , we arrive at
\begin{align}
	\frac{1}{4\ust}\ints(u-\ust)^2\leq \ints\left(u-\ust-\ust \ln\frac{u}{\ust}\right)\leq \frac{3}{4\ust}\ints(u-\ust)^2, \quad \forall\, t>t_0.\label{t.1.2.1}
\end{align}
Taking into account the inequality \eqref{t.1.2.1} it is clear that for all $t>t_0$ and with $\delta_1, \delta_2>0$, we have
\begin{align}
	\delta_1f_1(t)\leq \mathcal{E}_1(t)\leq \delta_2 f_1(t).\label{t.1.2.2}
\end{align}
From \eqref{l.4.3.6} and right-hand side of \eqref{t.1.2.2}, we get
\begin{align*}
	\dt\mathcal{E}_1(t)\leq- \zeta_1 f_1(t)\leq -\frac{\zeta_1}{\delta_2}\mathcal{E}_1(t), \qquad t>t_0.
\end{align*}
Solving the above differential inequality, gives
\begin{align*}
	\mathcal{E}_1(t)\leq C_1e^{-\frac{\zeta_1}{\delta_2}t}, \qquad t>t_0.
\end{align*}
where $C_1>0$. Now, left-hand side of \eqref{t.1.2.2} gives
\begin{align*}
	f_1(t)\leq \frac{1}{\delta_1} \mathscr{E}_1(t)\leq \frac{C_1}{\delta_1} e^{-\frac{\zeta_1}{\delta_2}t}, \qquad t>t_0.
\end{align*}
In view of \eqref{l.4.3.7} and then Lemma \ref{l.4.1}, finally we arrive at
\begin{align*}
	\big\|\ut-\ust\big\|_{\lis}+\big\|\vt-\vst\big\|_{\lis}\leq C_2e^{-\frac{\zeta_1}{\delta_2}t},
\end{align*}
for all $t>\tin$, with $C_2>0$. This completes the proof.\hfill \qedsymbol\\

{\bf Proof of Theorem \ref{t.3}.}
Choose $t_0>0$, we can apply a similar approach to the one used to derive \eqref{t.1.2.1} to obtain the following result
\begin{align}
	\frac{1}{2}\ints u^2\leq \ints u\leq 2\ints u^2+2\ints u, \quad \forall\, t>t_0.\label{t.1.3.1}
\end{align}
Considering the above inequality \eqref{t.1.3.1}, it is apparent that 
\begin{align}
	\delta_3f_2(t)\leq\mathcal{E}_2(t)\leq \delta_4\left(f_2(t)+\ints u\right), \quad t>\tin\label{t.1.3.2}
\end{align}
with $\delta_3, \delta_4>0$. From \eqref{l.4.4.4} and right-hand side of \eqref{t.1.3.2} with $ \zeta_2<\frac{a}{b}(f\mu-br)$, we obtain
\begin{align*}
	\dt\mathcal{E}_2(t)	\leq& -\frac{\zeta_2}{\delta_4}\mathcal{E}_2(t)-\Big(a(\mu\vbt-r)-\zeta_2\Big)\ints u\\
	\leq& -\frac{\zeta_2}{\delta_4}\mathcal{E}_2(t), \qquad t>t_0,
\end{align*}
such that $\mathcal{E}_2(t)\leq C_3e^{-\frac{\zeta_2}{\delta_4}t}$, $t>t_0$, where $C_3>0$. Now, by examining the left-hand side of \eqref{t.1.3.2}, we observe that
\begin{align*}
	f_2(t)\leq \frac{1}{\delta_3} \mathcal{E}_2(t)\leq \frac{C_3}{\delta_3}e^{-\frac{\zeta_2}{\delta_4}t}.
\end{align*}
By using a similar argument as in the proof of Theorem \ref{t.2}, we arrive at
\begin{align*}
	\big\|\ut\big\|_{\lis}+\big\|\vt-\vbt\big\|_{\lis}\leq C_4e^{-\frac{\zeta_2}{\delta_4}t},
\end{align*}
for some $t>\tin$, with $C_4>0$. This completes the proof.\hfill\qedsymbol\\

\section*{Conclusion}
In this work, we established the global existence of classical solutions to a chemotaxis system with lethal interactions, nonlinear diffusion and nonlinear productions in a bounded domain $\Omega \subset \mathbb{R}^n (n\geq 1)$ under appropriate conditions on the parameters $\alpha, \beta$ and $m$. Furthermore, we investigate the asymptotic behavior of solutions under specific parameter regimes through the construction and analysis of suitable Lyapunov functionals. Our analysis reveals two key stability outcomes with significant biological interpretations
\begin{enumerate}
	\item  If the interaction rate $\mu<\frac{br}{f_c}$, $2\beta\leq \alpha$ and $\chi^2$ is sufficiently small, the system admits a unique positive equilibrium state that is globally asymptotically stable. As a result, the \textit{E. coli} and $H_2O_2$ coexist in the long run for $m=1$ and $\kappa\geq 2$.
	
	\item When the interaction rate $\mu\geq \frac{br}{f_c}$, the semi-trivial equilibrium state becomes globally asymptotically stable, leading to the eventual extinction of the  \textit{E. coli} density  for $m=1$ and $\kappa=2$.
\end{enumerate}

\section*{Acknowledgments}
The authors wish to thank the anonymous referee for her/his careful reading of the original manuscript and their comments that eventually led to an improved presentation. 
GS and JS thank the Anusandhan National Research Foundation (ANRF), formerly Science and Engineering Research Board (SERB), Govt. of India for their support through Core Research Grant $(CRG/2023/001483)$ during this work.

\end{document}